\documentclass[10pt]{amsart}

\usepackage{amssymb,amsmath,amsthm}
\usepackage[all]{xy}
\usepackage{eucal}

\theoremstyle{plain}
\newtheorem{thm}[subsection]{Theorem}

\newtheorem{prop}[subsection]{Proposition}

\theoremstyle{definition}

\theoremstyle{remark}

\newcommand{\NN}{{ \mathbb{N} }}

\newcommand{\capX}{{ \mathcal{X} }}

\newcommand{\id}{{ \mathrm{id} }}

\newcommand{\SpaceZ}{{ \mathsf{S}_*^Z }}
\newcommand{\Space}{{ \mathsf{S}_* }}

\newcommand{\SpectraN}{{ \mathsf{Sp}^\NN }}

\newcommand{\Loop}{{ \Omega }}
\newcommand{\Loopt}{{ \tilde{\Omega} }}
\newcommand{\Susp}{{ \Sigma }}
\newcommand{\Suspt}{{ \tilde{\Sigma} }}

\newcommand{\wequiv}{{ \ \simeq \ }}

\newcommand{\Iso}{{  \ \cong \ }}

\DeclareMathOperator*{\holim}{holim}

\begin{document}

\title[Retractive spaces]{Retractive spaces and Bousfield-Kan completions}

\author{Zeshen Gu}
\author{John E. Harper}

\address{Department of Mathematics, The Ohio State University, 231 West 18th Ave, Columbus, OH 43210, USA}
\email{gu.1056@osu.edu}

\address{Department of Mathematics, The Ohio State University, Newark, 1179 University Dr, Newark, OH 43055, USA}
\email{harper.903@math.osu.edu}

\begin{abstract}
In this short paper we apply some recent techniques developed by Schonsheck, and subsequently Carr-Harper, in the context of operadic algebras in spectra---on convergence of  Bousfield-Kan completions and comparisons with convergence of the Taylor tower of the identity functor in Goodwillie's functor calculus---to the setting of retractive spaces: this arises when working with spaces centered away from the one-point space. Interestingly, in the retractive spaces context, the comparison results are stronger in terms of convergence outside of functor calculus' notion of ``radius of (strong) convergence'' for analytic functors. In particular, we give a new proof (and generalization to retractive spaces) of the Arone-Kankaanrinta result for convergence of the Taylor tower of the identity functor to various Bousfield-Kan completions; it's notable that no use is made of Snaith splittings---rather, we make extensive use of the kinds of homotopical estimates that appear in earlier work of Dundas and Dundas-Goodwillie-McCarthy.
\end{abstract}

\maketitle

\section{Introduction}

This paper is written simplicially so that ``space'' means ``simplicial set'' unless otherwise noted; see \cite{Bousfield_Kan, Goerss_Jardine}. In particular, we refer to the category of pointed simplicial sets $\Space$ as pointed spaces; this is equipped with the usual homotopy theory (\cite{Bousfield_Kan, Goerss_Jardine}). Our basic assumption is that $Z$ is a 0-connected fibrant pointed space. Denote by $\SpaceZ\Iso\id_Z\downarrow(\Space\downarrow Z)$ the factorization category (\cite[2.1]{Bauer_Johnson_McCarthy}, \cite[4.9]{Harper_Zhang}) of the identity map on $Z$, called the category of retractive pointed spaces over $Z$, equipped with the homotopy theory (\cite[2.1]{Bauer_Johnson_McCarthy}, \cite[4.9]{Harper_Zhang}) inherited from $\Space$; in particular, it has the structure of a simplicial cofibrantly generated model structure (\cite{Hirschhorn_overcategories}, \cite{Schwede_cotangent}) with an action of $\Space$ (\cite[4.2]{Hovey}). The setting of retractive spaces naturally arises in Goodwillie's homotopy functor calculus \cite{Goodwillie_calculus_2, Goodwillie_calculus_3} when working with Taylor towers centered away from the one-point space; see also \cite{Kuhn_survey}. When working with Bousfield-Kan completions, we make extensive use of the kinds of homotopical resolutions studied in \cite{Blumberg_Riehl}. We say that a retractive space $X$ over $Z$ is \emph{$k$-connected relative to $Z$} (or $k$-connected (rel. $Z$)) if the structure map $Z\rightarrow X$ is $k$-connected. 

Here are our main results. In the special case when $Z=*$ (the one-point space), Theorem \ref{MainTheorem_BK_completion} is proved by Bousfield and Hopkins \cite{Bousfield_cosimplicial_space} (for $r\geq 1$) and Carlsson \cite{Carlsson_equivariant} (for $r=\infty$) for any 0-connected nilpotent space $X$, and subsequently in \cite{Blomquist_Harper} for any 1-connected space $X$ (using different arguments closely related to \cite{Blomquist_Harper_integral_chains, Ching_Harper_derived, Dundas, Dundas_Goodwillie_McCarthy}). Our result in Theorem \ref{MainTheorem_BK_fibration}, generalizes this to any 0-connected (rel. $Z$) retractive space $F$ over $Z$, provided that, furthermore, $F$ fits into an appropriate homotopy pullback square. Our technical approach is motivated by the work in \cite{Schonsheck_fibration}, and the subsequent development in \cite{Carr_Harper}, for operadic algebras in spectra (where the estimates are different). We make extensive use of (the retractive version of) the homotopical estimates worked out in \cite{Blomquist_Harper}; these are the kinds of homotopical estimates that appear in earlier work of Dundas \cite{Dundas} and Dundas-Goodwillie-McCarthy \cite{Dundas_Goodwillie_McCarthy}. 

In the following theorems, $\Suspt_Z^r$ (resp. $\Loopt_Z^r$) denotes the derived $r$-fold suspension (\cite{Dwyer_Spalinski, Quillen}) (resp. derived $r$-fold loops (\cite{Dwyer_Spalinski, Quillen})) in $\SpaceZ$, and $\Suspt_Z^\infty$ (resp. $\Loopt_Z^\infty$) denotes derived stabilization (\cite{Hovey_spectra}) on $\SpaceZ$ (resp. derived 0-th object functor (\cite{Hovey_spectra}) on Hovey spectra $\SpectraN(\SpaceZ)$ on $\SpaceZ$).

\begin{thm}
\label{MainTheorem_BK_completion}
Assume that $Z$ is a 0-connected pointed space. Let $X$ be a retractive pointed space over $Z$. If $X$ is $1$-connected (rel. $Z$),  then the coaugmentations
\begin{align*}
  X\wequiv X^\wedge_{\Loopt_Z^r\Suspt_Z^r},\quad\quad (1\leq r\leq \infty)
\end{align*}
are weak equivalences in retractive pointed spaces over $Z$.
\end{thm}

\begin{thm}
\label{MainTheorem_BK_fibration}
Assume that $Z$ is a 0-connected pointed space. If $F\rightarrow X\rightarrow Y$ is a fibration sequence in retractive pointed spaces over $Z$ and $X,Y$ are 1-connected (rel. $Z$), then the coaugmentations
\begin{align*}
  F\wequiv F^\wedge_{\Loopt_Z^r\Suspt_Z^r},\quad\quad (1\leq r\leq \infty)
\end{align*} 
are weak equivalences in retractive pointed spaces over $Z$. More generally, let
\begin{align*}
\xymatrix{
  F\ar[r]\ar[d] & X\ar[d]\\
  A\ar[r] & Y
}
\end{align*}
be a homotopy pullback square in retractive pointed spaces over $Z$. If $A,X,Y$ are 1-connected (rel. $Z$), then the coaugmentations
\begin{align*}
  F\wequiv F^\wedge_{\Loopt_Z^r\Suspt_Z^r},\quad\quad (1\leq r\leq \infty)
\end{align*} 
are weak equivalences in retractive pointed spaces over $Z$.
\end{thm}

In the special case when $Z=*$, Theorem \ref{MainTheorem_exotic} is proved in Arone-Kankaanrinta \cite{Arone_Kankaanrinta}  for $r=\infty$ (using closely related, but different, arguments). We generalize their result to spaces centered away from $*$ and for $1\leq r\leq \infty$. Our technical approach is motivated by the work in \cite{Schonsheck_TQ}, and the subsequent development in \cite{Carr_Harper}, for operadic algebras in spectra (where the estimates are different). It's notable that no use is made of Snaith splittings---rather (as above) we make extensive use of (the retractive version of) the homotopical estimates worked out in \cite{Blomquist_Harper}, which are similar in spirit to the kinds of homotopical estimates appearing in the earlier work of  Dundas \cite{Dundas} and Dundas-Goodwillie-McCarthy \cite{Dundas_Goodwillie_McCarthy}; the possibility of giving a proof of Theorem \ref{MainTheorem_exotic} (when $Z=*$) along the lines developed here, was suggested in \cite{Schonsheck_private}.

\begin{thm}
\label{MainTheorem_exotic}
Assume that $Z$ is a 0-connected pointed space. Let $X$ be a retractive pointed space over $Z$. If $X$ is 0-connected (rel. $Z$), then there are weak equivalences of the form
\begin{align*}
  P_\infty^Z(\id)X\wequiv X^\wedge_{\Loopt_Z^r\Suspt_Z^r},
  \quad\quad (1\leq r\leq \infty)
\end{align*}
in retractive pointed spaces over $Z$; here, $P_n^Z(\id)X$ is the $n$-excisive approximation to the identity functor $\id$ on retractive pointed spaces over $Z$, evaluated at $X$, and $P_\infty^Z(\id)X$ denotes the homotopy limit of the associated Taylor tower $\{P_n^Z(\id)X\}$ of the identity functor $\id$, evaluated at $X$, in Goodwillie's functor calculus \cite{Goodwillie_calculus_3}.
\end{thm}

To keep this paper appropriately concise, we will freely use language from \cite{Blomquist_Harper}.

\section{Proofs of the main results}

To get Bousfield-Kan completion into the picture, we work with the kinds of homotopical resolutions studied in \cite{Blumberg_Riehl}. There are adjunctions of the form

\begin{align*}
\xymatrix{
  \SpaceZ\ar@<0.5ex>[r]^-{\Susp^r_Z} &
  \SpaceZ\ar@<0.5ex>[l]^-{\Loop^r_Z}
}
\quad\quad
\xymatrix{
  \SpaceZ\ar@<0.5ex>[r]^-{\Susp^\infty_Z} &
  \SpectraN(\SpaceZ)\ar@<0.5ex>[l]^-{\Loop^\infty_Z}
}
\quad\quad
(r\geq 1)
\end{align*}
with left adjoints on top, where $\Susp_Z^r$ is given by the pointed spaces action of $S^r:=(S^1)^{\wedge r}\in\Space$ on objects in $\SpaceZ$ and $\SpectraN(\SpaceZ)$ denotes Hovey spectra (\cite{Hovey_spectra}) on $\SpaceZ$; here, $\Susp_Z^\infty$ (resp. $\Loop_Z^\infty$) denotes the stabilization (resp. ``0-th object'') functor. Denote by $\id\rightarrow\Phi$ and $\Phi\Phi\rightarrow\Phi$ the unit and multiplication maps of the fibrant replacement monad $\Phi$ on $\SpaceZ$ (see \cite[6.1]{Blumberg_Riehl}) and define 
$\Loopt^r_Z:=\Loop^r_Z \Phi$. Similarly, denote by $\id\rightarrow F$ and $FF\rightarrow F$ the unit and multiplication maps of the fibrant replacement monad $F$ on $\SpectraN(\SpaceZ)$ (see \cite[6.1]{Blumberg_Riehl}) and define $\Loopt_Z^\infty:=\Loop^\infty_Z F$. Since every object in $\SpaceZ$ is cofibrant, $\Susp^r$ is already derived and we define $\Suspt^r:=\Susp^r$. If we iterate the comparison map $\id\rightarrow\Loopt^r_Z\Suspt^r_Z$ it follows that we can build a cosimplicial resolution of $\id$ with respect to $\Loopt^r_Z\Suspt^r_Z$ of the form
\begin{align}
\label{eq:derived_resolution}
\xymatrix{
  \id\ar[r] &
  (\Loopt^r_Z\Suspt^r_Z)\ar@<-0.5ex>[r]\ar@<0.5ex>[r] &
  (\Loopt^r_Z\Suspt^r_Z)^2
  \ar@<-1.0ex>[r]\ar[r]\ar@<1.0ex>[r] &
  (\Loopt^r_Z\Suspt^r_Z)^3\cdots
  }
\end{align}
for each $1\leq r\leq \infty$; these are the types of homotopical resolutions studied in \cite{Blumberg_Riehl}; see also \cite{Blomquist_Harper, Carr_Harper, Ching_Harper_derived}. Here, we are only showing the coface maps. If $X\in\SpaceZ$, the Bousfield-Kan completion of $X$ with respect to $\Loopt^r_Z\Suspt^r_Z$ is the homotopy limit
\begin{align}
  X^\wedge_{\Loopt^r_Z\Suspt^r_Z}:=\holim\nolimits_\Delta (\Loopt^r_Z\Suspt^r_Z)^{\bullet+1}(X)
\end{align}
of the Bousfield-Kan cosimplicial resolution \eqref{eq:derived_resolution} evaluated at $X$. To obtain the Bousfield-Kan completion tower, we filter $\Delta$ (\cite[5.22]{Blomquist_Harper}) by its subcategories $\Delta^{\leq n}\subset\Delta$, $n\geq 0$, and define
\begin{align*}
  (\Loopt^r_Z\Suspt^r_Z)_n:=
  \holim\nolimits_{\Delta^{\leq n}}(\Loopt^r_Z\Suspt^r_Z)^{\bullet+1},
  \quad\quad
  n\geq 0
\end{align*}
to obtain the $\Loopt^r_Z\Suspt^r_Z$-completion of $X$
\begin{align}
  X^\wedge_{\Loopt^r_Z\Suspt^r_Z}\wequiv\holim
  \Bigl(
  (\Loopt^r_Z\Suspt^r_Z)_0(X)\leftarrow(\Loopt^r_Z\Suspt^r_Z)_1(X)
  \leftarrow(\Loopt^r_Z\Suspt^r_Z)_2(X)\leftarrow\cdots
  \Bigr)
\end{align}
as the homotopy limit of the completion tower, where
\begin{align*}
  (\Loopt^r_Z\Suspt^r_Z)_0(X)\wequiv (\Loopt^r_Z\Suspt^r_Z)(X)
\end{align*}
 
For conceptual simplicity and convenience we denote by $*_Z:=Z$ the null object in $\SpaceZ$. It will be useful to denote by $*'_Z\wequiv *_Z$ an appropriately fattened-up version of the null object $*_Z$ in $\SpaceZ$.

\begin{prop}
\label{prop:uniformity_estimates_low_dimension}
Let $k\geq 0$ and $1\leq r\leq \infty$. Let $X$ be a retractive pointed space over $Z$. If $X$ is $k$-connected (rel. $Z$), then the comparison map $X\rightarrow\Loopt_Z^r\Suspt_Z^r X$ is $(2k+1)$-connected.
\end{prop}

\begin{proof}
Consider the case $r=1$. Consider a pushout cofibration 2-cube of the form
\begin{align*}
\xymatrix{
  X\ar[d]\ar[r] & {*'_Z}\ar[d]\\
  {*'_Z}\ar[r] & \Suspt_Z X
}
\end{align*}
in $\SpaceZ$. By assumption we know that the upper and left-hand 1-faces are $(k+1)$-connected. Since the 2-cube is $\infty$-cocartesian, it follows that the lower and right-hand 1-faces are $(k+1)$-connected. By higher Blakers-Massey \cite[2.5]{Goodwillie_calculus_2} for $\Space$, we know that the 2-cube is $l$-cartesian where $l$ is the minimum of 
\begin{align*}
  1-2+l_{\{1,2\}} &= -1 +\infty\\
  1-2+l_{\{1\}} + l_{\{2\}} &= -1 + (k+1) + (k+1)
\end{align*}
Hence $l=2k+1$, the 2-cube is $(2k+1)$-cartesian, and the comparison map $X\rightarrow\Loopt_Z\Suspt_Z X$ is $(2k+1)$-connected. The other cases ($r\geq 2$) follow by repeated application of the $r=1$ case to $\Suspt_Z X, \Suspt^2_Z X, \Suspt^3_Z X, \dots$ in the usual way, and finally, for $r=\infty$ by considering the homotopy colimit of the resulting sequence.
\end{proof}

For $k\geq 1$, this pattern persists for the iterative application of $\id\rightarrow\Loopt_Z^r\Suspt_Z^r$ to go from 0-cubes to 1-cubes to 2-cubes to 3-cubes, and so forth.

\begin{prop}
\label{prop:estimates_for_spaces}
Let $k\geq 1$ and $1\leq r\leq\infty$. Let $W$ be a finite set and $\capX$ a $W$-cube in $\SpaceZ$. Let $n=|W|$. If the $n$-cube $\capX$ is $(k(\id+1)+1)$-cartesian in $\SpaceZ$, then so is the $(n+1)$-cube of the form $\capX\rightarrow\Loopt_Z^r\Suspt_Z^r\capX$.
\end{prop}

\begin{proof}
These estimates are proved in \cite[1.7, 1.8]{Blomquist_Harper} for the special case of $Z=*$ using higher Blakers-Massey (and its dual) \cite[2.5, 2.6]{Goodwillie_calculus_2}, together with ideas closely related to \cite{Ching_Harper_derived, Dundas, Dundas_Goodwillie_McCarthy}; and similar to the proof of Proposition \ref{prop:uniformity_estimates_low_dimension}, exactly the same arguments (and estimates) remain true in the more general context of pointed spaces centered at $Z$; see, for instance, \cite{Carr_Harper} where the analogous passage to the retractive setting is demonstrated in detail for operadic algebras in spectra.
\end{proof}

\begin{proof}[Proof of Theorem \ref{MainTheorem_BK_completion}]
To verify that $X\wequiv X^\wedge_{\Loopt^r_Z\Suspt^r_Z}$, it suffices to verify that the map of the form $X\rightarrow (\Loopt^r_Z\Suspt^r_Z)_n(X)$ into the $n$-th stage of the Bousfield-Kan completion tower has connectivity strictly increasing with $n$. The connectivity of this map is the same as the cartesian-ness of the coface $(n+1)$-cube (\cite[3.13]{Blomquist_Harper_integral_chains}) of the coaugmented Bousfield-Kan cosimplicial resolution which we calculated (Propositions \ref{prop:uniformity_estimates_low_dimension} and \ref{prop:estimates_for_spaces}) to be $(((n+1)+1)+1)=n+3$, which completes the proof.
\end{proof}

\begin{prop}
\label{prop:estimates_for_exotic_proof_r_equals_1}
Let $n\geq 1$. Let $X$ be a retractive pointed space over $Z$. If $X$ is 0-connected (rel. $Z$), then the maps
\begin{align*}
  (\Loopt_Z\Suspt_Z)^k (X)\xrightarrow{(*)_n}
  P_n^Z(\Loopt_Z\Suspt_Z)^k (X),\quad\quad
  k\geq 1
\end{align*}
are $(n+1)$-connected.
\end{prop}

\begin{proof}
Consider the case of $n=1$. The 1-excisive approximation $P_1^Z(\Loopt_Z\Suspt_Z)^k(X)$ to the functor $(\Loopt_Z\Suspt_Z)^k$ on retractive pointed spaces over $Z$, evaluated at $X$, is the homotopy colimit of
\begin{align*}
  (\Loopt_Z\Suspt_Z)^k(X)\xrightarrow{(\#)_1}
  T_1^Z(\Loopt_Z\Suspt_Z)^k(X)\rightarrow
  T_1^Z(T_1^Z(\Loopt_Z\Suspt_Z)^k)(X)\rightarrow\cdots
\end{align*}
the indicated sequence (\cite[Section 1]{Goodwillie_calculus_3}); we want to estimate the connectivities of these maps. It follows, by iteratively applying higher Blakers-Massey (and its dual) \cite[2.5, 2.6]{Goodwillie_calculus_2} for $\Space$ that the maps $(\#)_1$ are 2-connected, and the other maps are higher connected. In more detail: here is the basic idea for the maps $(\#)_1$; estimates for the other (more highly connected) maps are similar. Consider the $\infty$-cocartesian 2-cube $\capX$ of the form
\begin{align*}
\xymatrix{
  X\ar[d]\ar[r] & {*'_Z}\ar[d]\\
  {*'_Z}\ar[r] & \Suspt_Z X
}
\end{align*}
in $\SpaceZ$. Since $X$ is 0-connected (rel. $Z$), we know that $*_Z\rightarrow X$ is 0-connected and hence $X\rightarrow *_Z$ is 1-connected (\cite[1.5]{Goodwillie_calculus_2}); therefore we know that $\capX$ satisfies: the 1-subcubes are 1-connected and the 2-cube is $\infty$-cocartesian. Then $\Suspt_Z\capX$ satisfies: the 1-subcubes are 2-connected and the 2-subcubes (there is only one) are $\infty$-cocartesian. By higher Blakers-Massey \cite[2.5]{Goodwillie_calculus_2} for $\Space$, we know the 2-cube is $k$-cartesian where $k$ is the minimum of
\begin{align*}
  1-2 + k_{\{1,2\}} &= -1 + \infty\\
  1-2 + k_{\{1\}} + k_{\{2\}} &= -1+2+2
\end{align*}
Hence $k=3$, our 2-cube is 3-cartesian, and $\Suspt_Z\capX$ satisfies: the 1-subcubes are 2-connected and the 2-subcubes are 3-cartesian. Then $(\Loopt_Z\Suspt_Z)\capX$ satisfies: the 1-subcubes are 1-connected and the 2-subcubes are 2-cartesian. Hence we have verifed that the map
\begin{align*}
  (\Loopt_Z\Suspt_Z)(X)\xrightarrow{(\#)_1}
  T_1^Z(\Loopt_Z\Suspt_Z)(X)
\end{align*}
is 2-connected. Let's keep going. By higher dual Blakers-Massey \cite[2.6]{Goodwillie_calculus_2} for $\Space$, we know the 2-cube is $k$-cocartesian where $k$ is the minimum of
\begin{align*}
  2-1+k_{\{1,2\}} &= 1+2\\
  2-1+k_{\{1\}}+k_{\{2\}} &= 1+1+1 
\end{align*}
Hence $k=3$, our 2-cube is 3-cocartesian, and $(\Loopt_Z\Suspt_Z)\capX$ satisfies: the 1-subcubes are 1-connected and the 2-subcubes are 3-cocartesian. Then $\Suspt_Z(\Loopt_Z\Suspt_Z)\capX$ satisfies: the 1-subcubes are 2-connected and the 2-subcubes are 4-cocartesian. By higher Blakers-Massey \cite[2.5]{Goodwillie_calculus_2} for $\Space$, we know the 2-cube is $k$-cartesian where $k$ is the minimum of
\begin{align*}
  1-2 + k_{\{1,2\}} &= -1 + 4\\
  1-2 + k_{\{1\}} + k_{\{2\}} &= -1+2+2
\end{align*}
Hence, $k=3$, our 2-cube is 3-cartesian, and $\Suspt_Z(\Loopt_Z\Suspt_Z)\capX$ satisfies: the 1-subcubes are 2-connected and the 2-subcubes are 3-cartesian. Then $(\Loopt_Z\Suspt_Z)^2\capX$ satisfies: the 1-subcubes are 1-connected and the 2-subcubes are 2-cartesian. Hence we have verifed that the map
\begin{align*}
  (\Loopt_Z\Suspt_Z)^2(X)\xrightarrow{(\#)_1}
  T_1^Z(\Loopt_Z\Suspt_Z)^2(X)
\end{align*}
is 2-connected. Let's keep going. By higher dual Blakers-Massey \cite[2.6]{Goodwillie_calculus_2} for $\Space$, we know the 2-cube is $k$-cocartesian where $k$ is the minimum of
\begin{align*}
  2-1+k_{\{1,2\}} &= 1+2\\
  2-1+k_{\{1\}}+k_{\{2\}} &= 1+1+1 
\end{align*}
Hence $k=3$, our 2-cube is 3-cocartesian, and $(\Loopt_Z\Suspt_Z)^2\capX$ satisfies: the 1-subcubes are 1-connected and the 2-subcubes are 3-cocartesian. Then $\Suspt_Z(\Loopt_Z\Suspt_Z)^2\capX$ satisfies: the 1-subcubes are 2-connected and the 2-subcubes are 4-cocartesian. By higher Blakers-Massey \cite[2.5]{Goodwillie_calculus_2} for $\Space$, we know the 2-cube is $k$-cartesian where $k$ is the minimum of
\begin{align*}
  1-2 + k_{\{1,2\}} &= -1 + 4\\
  1-2 + k_{\{1\}} + k_{\{2\}} &= -1+2+2
\end{align*}
Hence, $k=3$, our 2-cube is 3-cartesian, and $\Suspt_Z(\Loopt_Z\Suspt_Z)^2\capX$ satisfies: the 1-subcubes are 2-connected and the 2-subcubes are 3-cartesian. Then $(\Loopt_Z\Suspt_Z)^3\capX$ satisfies: the 1-subcubes are 1-connected and the 2-subcubes are 2-cartesian. Hence we have verifed that the map
\begin{align*}
  (\Loopt_Z\Suspt_Z)^3(X)\xrightarrow{(\#)_1}
  T_1^Z(\Loopt_Z\Suspt_Z)^3(X)
\end{align*}
is 2-connected; notice how the subcube estimates have stabilized at each respective step. And so forth. Hence it follows that the maps
\begin{align*}
  (\Loopt_Z\Suspt_Z)^k(X)\xrightarrow{(\#)_1}
  T_1^Z(\Loopt_Z\Suspt_Z)^k(X),\quad\quad
  k\geq 1
\end{align*}
are 2-connected. Consider the case of $n=2$. The 2-excisive approximation $P_2^Z(\Loopt_Z\Suspt_Z)^k(X)$ to the functor $(\Loopt_Z\Suspt_Z)^k$ on retractive pointed spaces over $Z$, evaluated at $X$, is the homotopy colimit of
\begin{align*}
  (\Loopt_Z\Suspt_Z)^k(X)\xrightarrow{(\#)_2}
  T_2^Z(\Loopt_Z\Suspt_Z)^k(X)\rightarrow
  T_2^Z(T_2^Z(\Loopt_Z\Suspt_Z)^k)(X)\rightarrow\cdots
\end{align*}
the indicated sequence (\cite[Section 1]{Goodwillie_calculus_3}); we want to estimate the connectivities of these maps. It follows, by iteratively applying higher Blakers-Massey (and its dual) \cite[2.5, 2.6]{Goodwillie_calculus_2} for $\Space$ that the maps $(\#)_2$ are 3-connected, and the other maps are higher connected. In more detail: here is the basic idea for the maps $(\#)_2$; estimates for the other (more highly connected) maps are similar. Consider a strongly $\infty$-cocartesian 3-cube $\capX$ satisfying: the 1-subcubes are 1-connected, the 2-subcubes are $\infty$-cocartesian, and the 3-subcubes (there is only one) are $\infty$-cocartesian. Then $\Suspt_Z\capX$ satisfies: the 1-subcubes are 2-connected, the 2-subcubes are $\infty$-cocartesian, and the 3-subcubes are $\infty$-cocartesian. By higher Blakers-Massey \cite[2.5]{Goodwillie_calculus_2} for $\Space$, we know the 3-cube is $k$-cartesian where $k$ is the minimum of
\begin{align*}
  1-3+k_{\{1,2,3\}} &= -2+\infty \\
  1-3+k_{\{1,2\}}+k_{\{3\}} &= -2+\infty+2\\
  1-3+k_{\{1\}}+k_{\{2\}}+k_{\{3\}} &= -2+2+2+2
\end{align*}
Hence $k=4$, our 3-cube is 4-cartesian, and $\Suspt_Z\capX$ satisfies: the 1-subcubes are 2-connected, the 2-subcubes are 3-cartesian, and the 3-subcubes are 4-cartesian. Then $(\Loopt_Z\Suspt_Z)\capX$ satisfies: the 1-subcubes are 1-connected, the 2-subcubes are 2-cartesian, and the 3-subcubes are 3-cartesian. Hence we have verifed that the map
\begin{align*}
  (\Loopt_Z\Suspt_Z)(X)\xrightarrow{(\#)_2}
  T_2^Z(\Loopt_Z\Suspt_Z)(X)
\end{align*}
is 3-connected. Let's keep going. By higher dual Blakers-Massey \cite[2.6]{Goodwillie_calculus_2} for $\Space$, we know the 3-cube is $k$-cocartesian where $k$ is the minimum of
\begin{align*}
  3-1+k_{\{1,2,3\}} &=  2+3\\
  3-1+k_{\{1,2\}}+k_{\{3\}} &= 2+2+1\\
  3-1+k_{\{1\}}+k_{\{2\}}+k_{\{3\}} &=2+1+1+1 
\end{align*}
Hence $k=5$, our 3-cube is 5-cocartesian, and $(\Loopt_Z\Suspt_Z)\capX$ satisfies: the 1-subcubes are 1-connected, the 2-subcubes are 3-cocartesian, and the 3-subcubes are 5-cocartesian. Then $\Suspt_Z(\Loopt_Z\Suspt_Z)\capX$ satisfies: the 1-subcubes are 2-connected, the 2-subcubes are 4-cocartesian, and the 3-subcubes are 6-cocartesian. By higher Blakers-Massey \cite[2.5]{Goodwillie_calculus_2} for $\Space$, we know the 3-cube is $k$-cartesian where $k$ is the minimum of
\begin{align*}
  1-3+k_{\{1,2,3\}} &= -2+6 \\
  1-3+k_{\{1,2\}}+k_{\{3\}} &= -2+4+2\\
  1-3+k_{\{1\}}+k_{\{2\}}+k_{\{3\}} &= -2+2+2+2
\end{align*}
Hence, $k=4$, our 3-cube is 4-cartesian, and $\Suspt_Z(\Loopt_Z\Suspt_Z)\capX$ satisfies: the 1-subcubes are 2-connected, the 2-subcubes are 3-cartesian, and the 3-subcubes are 4-cartesian. Then $(\Loopt_Z\Suspt_Z)^2\capX$ satisfies: the 1-subcubes are 1-connected, the 2-subcubes are 2-cartesian, and the 3-subcubes are 3-cartesian. Hence we have verifed that the map
\begin{align*}
  (\Loopt_Z\Suspt_Z)^2(X)\xrightarrow{(\#)_2}
  T_2^Z(\Loopt_Z\Suspt_Z)^2(X)
\end{align*}
is 3-connected. Let's keep going. By higher dual Blakers-Massey \cite[2.6]{Goodwillie_calculus_2} for $\Space$, we know the 3-cube is $k$-cocartesian where $k$ is the minimum of
\begin{align*}
  3-1+k_{\{1,2,3\}} &=  2+3\\
  3-1+k_{\{1,2\}}+k_{\{3\}} &= 2+2+1\\
  3-1+k_{\{1\}}+k_{\{2\}}+k_{\{3\}} &=2+1+1+1 
\end{align*}
Hence $k=5$, our 3-cube is 5-cocartesian, and $(\Loopt_Z\Suspt_Z)^2\capX$ satisfies: the 1-subcubes are 1-connected, the 2-subcubes are 3-cocartesian, and the 3-subcubes are 5-cocartesian. Then $\Suspt_Z(\Loopt_Z\Suspt_Z)^2\capX$ satisfies: the 1-subcubes are 2-connected, the 2-subcubes are 4-cocartesian, and the 3-subcubes are 6-cocartesian. By higher Blakers-Massey \cite[2.5]{Goodwillie_calculus_2} for $\Space$, we know the 3-cube is $k$-cartesian where $k$ is the minimum of
\begin{align*}
  1-3+k_{\{1,2,3\}} &= -2+6 \\
  1-3+k_{\{1,2\}}+k_{\{3\}} &= -2+4+2\\
  1-3+k_{\{1\}}+k_{\{2\}}+k_{\{3\}} &= -2+2+2+2
\end{align*}
Hence, $k=4$, our 3-cube is 4-cartesian, and $\Suspt_Z(\Loopt_Z\Suspt_Z)^2\capX$ satisfies: the 1-subcubes are 2-connected, the 2-subcubes are 3-cartesian, and the 3-subcubes are 4-cartesian. Then $(\Loopt_Z\Suspt_Z)^3\capX$ satisfies: the 1-subcubes are 1-connected, the 2-subcubes are 2-cartesian, and the 3-subcubes are 3-cartesian. Hence we have verifed that the map
\begin{align*}
  (\Loopt_Z\Suspt_Z)^3(X)\xrightarrow{(\#)_2}
  T_2^Z(\Loopt_Z\Suspt_Z)^3(X)
\end{align*}
is 3-connected; notice how the subcube estimates have stabilized at each respective step. And so forth. Hence it follows that the maps
\begin{align*}
  (\Loopt_Z\Suspt_Z)^k(X)\xrightarrow{(\#)_2}
  T_2^Z(\Loopt_Z\Suspt_Z)^k(X),\quad\quad
  k\geq 1
\end{align*}
are 3-connected. And so forth.
\end{proof}

\begin{prop}
\label{prop:estimates_for_exotic}
Let $n\geq 1$ and $1\leq r\leq\infty$. Let $X$ be a retractive pointed space over $Z$. If $X$ is 0-connected (rel. $Z$), then the maps
\begin{align*}
  (\Loopt^r_Z\Suspt^r_Z)^k (X)\xrightarrow{(*)_n}
  P_n^Z(\Loopt^r_Z\Suspt^r_Z)^k (X),\quad\quad
  k\geq 1
\end{align*}
are $(n+1)$-connected.
\end{prop}

\begin{proof}
A detailed proof of the $r=1$ case is given above (Proposition \ref{prop:estimates_for_exotic_proof_r_equals_1}), and the other cases are similar. In the case of $r=\infty$, several of the steps are easier since $\Suspt_Z^\infty$ preserves cocartesian-ness, $\Loopt_Z^\infty$ preserves cartesian-ness, and the stable estimates in \cite[3.10]{Ching_Harper} are available for each estimate step following the application of $\Suspt_Z^\infty$.
\end{proof}

\begin{prop}
\label{prop:connectivities_eventually_increase}
Let $n\geq 1$. Let $X$ be a retractive pointed space over $Z$. If $X$ is 0-connected (rel. $Z$), then the maps
\begin{align*}
  (\Loopt_Z\Suspt_Z)_k (X)\xrightarrow{(**)_n}
  P_{n+k}^Z(\Loopt_Z\Suspt_Z)_k (X),\quad\quad
  k\geq 0
\end{align*}
are $(n+1)$-connected.
\end{prop}

\begin{proof}
Consider the case of $k=0$. Then the map $(**)_n$ is $(n+1)$-connected by Proposition \ref{prop:estimates_for_exotic}. Consider the case of $k=1$. By definition, $(\Loopt_Z\Suspt_Z)_1 X$ fits into an $\infty$-cartesian 2-cube of the form (\cite[5.26]{Blomquist_Harper})
\begin{align*}
\xymatrix{
  (\Loopt_Z\Suspt_Z)_1 X\ar[d]\ar[r] & (\Loopt_Z\Suspt_Z) X\ar[d]\\
  (\Loopt_Z\Suspt_Z) X\ar[r] & (\Loopt_Z\Suspt_Z)^2 X
}
\end{align*}
and therefore the map $(**)_n$ fits into a 3-cube of the form
\begin{align}
\label{eq:nice_3_cube_picture}
\xymatrix{
(\Loopt_Z\Suspt_Z)_1 X\ar[dd]\ar[rr]\ar[dr]^-{(**)_n} &&
(\Loopt_Z\Suspt_Z) X\ar'[d][dd]\ar[dr]^-{(*)_{n+1}}\\
&P_{n+1}^Z(\Loopt_Z\Suspt_Z)_1 X\ar[dd]\ar[rr] &&
P_{n+1}^Z(\Loopt_Z\Suspt_Z) X\ar[dd]\\
(\Loopt_Z\Suspt_Z) X\ar'[r][rr]\ar[dr]^-{(*)_{n+1}} &&
(\Loopt_Z\Suspt_Z)^2 X\ar[dr]^-{(*)_{n+1}}\\
&P_{n+1}^Z(\Loopt_Z\Suspt_Z) X\ar[rr] &&
P_{n+1}^Z(\Loopt_Z\Suspt_Z)^2 X
}
\end{align}
Several applications of \cite[1.6]{Goodwillie_calculus_2} show that the map $(**)_n$ is $(n+1)$-connected. In more detail: the back 2-face is $\infty$-cartesian, hence the front 2-face is $\infty$-cartesian (\cite[1.7]{Goodwillie_calculus_3}). Therefore, the 3-cube is $\infty$-cartesian. By Proposition \ref{prop:estimates_for_exotic}, the maps $(*)_{n+1}$ are $(n+2)$-connected, hence the right-hand 2-face is $(n+1)$-cartesian. Since the 3-cube is $\infty$-cartesian, we therefore know the left-hand 2-face is $(n+1)$-cartesian; and hence, since the map $(*)_{n+1}$ is $(n+2)$-connected, therefore we know the map $(**)_n$ is $(n+1)$-connected. Consider the case of $k=2$. By definition, $(\Loopt_Z\Suspt_Z)_2 X$ fits into an $\infty$-cartesian 3-cube $\capX$ of the form

\begin{align*}
\xymatrix{
(\Loopt_Z\Suspt_Z)_2 X\ar[dd]\ar[rr]\ar[dr] &&
(\Loopt_Z\Suspt_Z) X\ar'[d][dd]\ar[dr]\\
&(\Loopt_Z\Suspt_Z) X\ar[dd]\ar[rr] &&
(\Loopt_Z\Suspt_Z)^2 X\ar[dd]\\
(\Loopt_Z\Suspt_Z) X\ar'[r][rr]\ar[dr] &&
(\Loopt_Z\Suspt_Z)^2 X\ar[dr]\\
&(\Loopt_Z\Suspt_Z)^2 X\ar[rr] &&
(\Loopt_Z\Suspt_Z)^3 X
}
\end{align*}
and therefore the map $(**)_n$ fits into a 4-cube of the form $\capX\rightarrow P_{n+2}^Z\capX$. Several applications of \cite[1.6]{Goodwillie_calculus_2} show that the map $(**)_n$ is $(n+1)$-connected. In more detail: $\capX$ is $\infty$-cartesian, hence  $P_{n+2}^Z\capX$ is $\infty$-cartesian (\cite[1.7]{Goodwillie_calculus_3}). Therefore, the 4-cube is $\infty$-cartesian. Consider the 3-face of the form

\begin{align*}
\xymatrix{
(\Loopt_Z\Suspt_Z) X\ar[dd]\ar[rr]^-{(*)_{n+2}}\ar[dr] &&
P_{n+2}^Z(\Loopt_Z\Suspt_Z) X\ar'[d][dd]\ar[dr]\\
&(\Loopt_Z\Suspt_Z)^2 X\ar[dd]\ar[rr]^(0.3){(*)_{n+2}} &&
P_{n+2}^Z(\Loopt_Z\Suspt_Z)^2 X\ar[dd]\\
(\Loopt_Z\Suspt_Z)^2 X\ar'[r][rr]^-{(*)_{n+2}}\ar[dr] &&
P_{n+2}^Z(\Loopt_Z\Suspt_Z)^2 X\ar[dr]\\
&(\Loopt_Z\Suspt_Z)^3 X\ar[rr]^-{(*)_{n+2}} &&
P_{n+2}^Z(\Loopt_Z\Suspt_Z)^3 X
}
\end{align*}
By Proposition \ref{prop:estimates_for_exotic}, the maps $(*)_{n+2}$ are $(n+3)$-connected, hence the top and bottom 2-faces are $(n+2)$-cartesian, and therefore the 3-face is $(n+1)$-cartesian. Since the 4-cube is $\infty$-cartesian, we therefore know the opposite 3-face of the form
\begin{align*}
\xymatrix{
(\Loopt_Z\Suspt_Z)_2 X\ar[dd]\ar[rr]^-{(**)_n}\ar[dr] &&
P_{n+2}^Z(\Loopt_Z\Suspt_Z)_2 X\ar'[d][dd]\ar[dr]\\
&(\Loopt_Z\Suspt_Z) X\ar[dd]\ar[rr]^(0.3){(*)_{n+2}} &&
P_{n+2}^Z(\Loopt_Z\Suspt_Z) X\ar[dd]\\
(\Loopt_Z\Suspt_Z) X\ar'[r][rr]^-{(*)_{n+2}}\ar[dr] &&
P_{n+2}^Z(\Loopt_Z\Suspt_Z) X\ar[dr]\\
&(\Loopt_Z\Suspt_Z)^2 X\ar[rr]^-{(*)_{n+2}} &&
P_{n+2}^Z(\Loopt_Z\Suspt_Z)^2 X
}
\end{align*}
is $(n+1)$-cartesian; and hence, since the maps $(*)_{n+2}$ are $(n+3)$-connected, we know the bottom 2-face is $(n+2)$-cartesian, it follows that the top face is $(n+1)$-cartesian, and since the map $(*)_{n+2}$ is $(n+3)$-connected, therefore we know the map $(**)_n$ is $(n+1)$-connected. The other cases similarly follow by repeated applications of \cite[1.6]{Goodwillie_calculus_2}.
\end{proof}

\begin{prop}
\label{prop:connectivities_eventually_increase_general_r}
Let $n\geq 1$ and $1\leq r\leq \infty$. Let $X$ be a retractive pointed space over $Z$. If $X$ is 0-connected (rel. $Z$), then the maps
\begin{align*}
  (\Loopt_Z^r\Suspt_Z^r)_k (X)\xrightarrow{(**)_n}
  P_{n+k}^Z(\Loopt_Z^r\Suspt_Z^r)_k (X),\quad\quad
  k\geq 0
\end{align*}
are $(n+1)$-connected.
\end{prop}

\begin{proof}
A detailed proof of the $r=1$ case is given above (Proposition \ref{prop:connectivities_eventually_increase}), and the other cases are similar; the estimates are identical (Proposition \ref{prop:estimates_for_exotic}).
\end{proof}

\begin{proof}[Proof of Theorem \ref{MainTheorem_exotic}]
We follow the basic proof strategy in \cite{Schonsheck_TQ}, and the subsequent development in \cite{Carr_Harper}, for operadic algebras in spectra (where the estimates are different). Here is the basic idea.  Consider the case of $r=1$. We start with the Bousfield-Kan completion tower of the form
\begin{align*}
\xymatrix{
  (\Loopt_Z\Suspt_Z)_0 &
  (\Loopt_Z\Suspt_Z)_1\ar[l] &
  (\Loopt_Z\Suspt_Z)_2\ar[l]\cdots
}
\end{align*}
and resolve each term by its Taylor tower to produce the tower of towers diagram
\begin{align}
\label{eq:tower_of_towers_exotic}
\xymatrix{
  P_3^Z(\Loopt_Z\Suspt_Z)_0(X)\ar[d] &
  P_3^Z(\Loopt_Z\Suspt_Z)_1(X)\ar[l]\ar[d] &
  P_3^Z(\Loopt_Z\Suspt_Z)_2(X)\ar[l]\ar[d]\cdots\\
  P_2^Z(\Loopt_Z\Suspt_Z)_0(X)\ar[d] &
  P_2^Z(\Loopt_Z\Suspt_Z)_1(X)\ar[l]\ar[d] &
  P_2^Z(\Loopt_Z\Suspt_Z)_2(X)\ar[l]\ar[d]\cdots\\
  P_1^Z(\Loopt_Z\Suspt_Z)_0(X) &
  P_1^Z(\Loopt_Z\Suspt_Z)_1(X)\ar[l] &
  P_1^Z(\Loopt_Z\Suspt_Z)_2(X)\ar[l]\cdots
}
\end{align}
By our uniformity estimates (Propositions \ref{prop:uniformity_estimates_low_dimension} and \ref{prop:estimates_for_spaces}), it follows immediately that $\id\rightarrow(\Loopt_Z\Suspt_Z)_n$ satisfies $O_{n+1}$ (\cite[1.2]{Goodwillie_calculus_3}) for each $n\geq 0$; in other words, via this map the functors $\id$ and $(\Loopt_Z\Suspt_Z)_n$ agree to order $n+1$ and hence the maps
\begin{align*}
  P_{m}^Z(\id)(X)&\xrightarrow{\wequiv} P_{m}^Z(\Loopt_Z\Suspt_Z)_n(X),
  \quad\quad
  1\leq m\leq n+1
\end{align*}
are weak equivalences for every $n\geq 0$. It follows that $\holim$ applied horizontally produces the Taylor tower of the identity functor $\{P_n^Z(\id)\}$ and hence applying $\holim$ first horizontally and then vertically produces
\begin{align*}
  \holim_{\mathsf{vert}}\holim_{\mathsf{horiz}}
  \eqref{eq:tower_of_towers_exotic}\wequiv P_\infty^Z(\id)(X)
\end{align*}
What about the other way? By our estimates (Proposition \ref{prop:connectivities_eventually_increase_general_r}), it follows immediately that $\holim$ applied vertically produces the Bousfield-Kan completion tower
\begin{align*}
\xymatrix{
  (\Loopt_Z\Suspt_Z)_0(X) &
  (\Loopt_Z\Suspt_Z)_1(X)\ar[l] &
  (\Loopt_Z\Suspt_Z)_2(X)\ar[l]\cdots
}
\end{align*}
and hence applying $\holim$ first vertically and then horizontally produces
\begin{align*}
  \holim_{\mathsf{horiz}}\holim_{\mathsf{vert}}
  \eqref{eq:tower_of_towers_exotic}\wequiv X^\wedge_{\Loopt_Z\Suspt_Z}
\end{align*}
Hence we have verified that
\begin{align*}
  P_\infty^Z(\id)(X)\wequiv X^\wedge_{\Loopt_Z\Suspt_Z}
\end{align*}
The other cases are similar; the estimates are identical (Proposition \ref{prop:connectivities_eventually_increase_general_r}).
\end{proof}

\begin{prop}
\label{prop:preserving_id_plus_one_cartesian_cubes_r_equals_one}
If $n\geq -1$, then $(\Loopt_Z\Suspt_Z)$ preserves $(\id+1)$-cartesian $(n+1)$-cubes in $\SpaceZ$.
\end{prop}

\begin{proof}
The cases for $n=-1,0$ are trivial. Now that we know the desired behavior is satisfied on 0-subcubes, we will not continue to indicate their estimates below when verifying the $(\id+1)$-cartesian property. Consider the case of $n=1$. Assume that $\capX$ is an $(\id+1)$-cartesian 2-cube in $\SpaceZ$. Then $\capX$ satisfies: the 1-subcubes are 2-connected, and the 2-subcubes (there is only one) are 3-cartesian. By higher dual Blakers-Massey \cite[2.6]{Goodwillie_calculus_2} for $\Space$, we know the 2-cube is $k$-cocartesian where $k$ is the minimum of
\begin{align*}
  2-1+k_{\{1,2\}} &= 1+3\\
  2-1+k_{\{1\}}+k_{\{2\}} &= 1+2+2 
\end{align*}
Hence $k=4$, our 3-cube is 4-cocartesian, and $\capX$ satisfies: the 1-subcubes are 2-connected, and the 2-subcubes are 4-cocartesian. Then $\Suspt_Z\capX$ satisfies: the 1-subcubes are 3-connected, and the 2-subcubes are 5-cocartesian. By higher Blakers-Massey \cite[2.5]{Goodwillie_calculus_2} for $\Space$, we know the 2-cube is $k$-cartesian where $k$ is the minimum of
\begin{align*}
  1-2 + k_{\{1,2\}} &= -1 + 5\\
  1-2 + k_{\{1\}} + k_{\{2\}} &= -1+3+3
\end{align*}
Hence, $k=4$, our 2-cube is 4-cartesian, and $\Suspt_Z\capX$ satisfies: the 1-subcubes are 3-connected, and the 2-subcubes are 4-cartesian. Then $\Loopt_Z\Suspt_Z\capX$ satisfies: the 1-subcubes are 2-connected, and the 2-subcubes are 3-cartesian. Hence we have verified that $(\Loopt_Z\Suspt_Z)\capX$ is $(\id+1)$-cartesian in $\SpaceZ$. Consider the case of $n=2$. Assume that $\capX$ is an $(\id+1)$-cartesian 3-cube in $\SpaceZ$. Then $\capX$ satisfies: the 1-subcubes are 2-connected, the 2-subcubes are 3-cartesian, and the 3-subcubes (there is only one) are 4-cartesian. By higher dual Blakers-Massey \cite[2.5]{Goodwillie_calculus_2} for $\Space$, we know the 3-cube is $k$-cocartesian where $k$ is the minimum of
\begin{align*}
  3-1+k_{\{1,2,3\}} &=  2+4\\
  3-1+k_{\{1,2\}}+k_{\{3\}} &= 2+3+2\\
  3-1+k_{\{1\}}+k_{\{2\}}+k_{\{3\}} &=2+2+2+2 
\end{align*}
Hence $k=6$, our 3-cube is 6-cocartesian, and $\capX$ satisfies: the 1-subcubes are 2-connected, the 2-subcubes are 4-cocartesian, and the 3-subcubes are 6-cocartesian. Then $\Suspt_Z\capX$ satisfies: the 1-subcubes are 3-connected, the 2-subcubes are 5-cocartesian, and the 3-subcubes are 7-cocartesian. By higher Blakers-Massey \cite[2.5]{Goodwillie_calculus_2} for $\Space$, we know the 3-cube is $k$-cartesian where $k$ is the minimum of
\begin{align*}
  1-3+k_{\{1,2,3\}} &= -2+7 \\
  1-3+k_{\{1,2\}}+k_{\{3\}} &= -2+5+3\\
  1-3+k_{\{1\}}+k_{\{2\}}+k_{\{3\}} &= -2+3+3+3
\end{align*}
Hence, $k=5$, our 3-cube is 5-cartesian, and $\Suspt_Z\capX$ satisfies: the 1-subcubes are 3-connected, the 2-subcubes are 4-cartesian, and the 3-subcubes are 5-cartesian. Then $(\Loopt_Z\Suspt_Z)\capX$ satisfies: the 1-subcubes are 2-connected, the 2-subcubes are 3-cartesian, and the 3-subcubes are 4-cartesian. Hence we have verified that $(\Loopt_Z\Suspt_Z)\capX$ is $(\id+1)$-cartesian in $\SpaceZ$. And so forth.
\end{proof}

\begin{prop}
\label{prop:preserving_id_plus_one_cartesian_cubes}
Let $1\leq r\leq \infty$. If $n\geq -1$, then $(\Loopt^r_Z\Suspt^r_Z)$ preserves $(\id+1)$-cartesian $(n+1)$-cubes in $\SpaceZ$.
\end{prop}

\begin{proof}
A detailed proof of the $r=1$ case is given above (Proposition \ref{prop:preserving_id_plus_one_cartesian_cubes_r_equals_one}), and the other cases are similar. In the case of $r=\infty$, several of the steps are easier since $\Suspt_Z^\infty$ preserves cocartesian-ness, $\Loopt_Z^\infty$ preserves cartesian-ness, and the stable estimates in \cite[3.10]{Ching_Harper} are available for each estimate step following the application of $\Suspt_Z^\infty$.
\end{proof}

\begin{proof}[Proof of Theorem \ref{MainTheorem_BK_fibration}]
We follow the basic proof strategy in \cite{Schonsheck_fibration}, and the subsequent development in \cite{Carr_Harper}, for operadic algebras in spectra (where the estimates are different). Here is the basic idea. Consider the case $r=1$. Start with the Bousfield-Kan cosimplicial resolution
\begin{align}
\label{eq:a_nice_resolution}
\xymatrix{
  \id\ar[r] &
  (\Loopt_Z\Suspt_Z)\ar@<-0.5ex>[r]\ar@<0.5ex>[r] &
  (\Loopt_Z\Suspt_Z)^2
  \ar@<-1.0ex>[r]\ar[r]\ar@<1.0ex>[r] &
  (\Loopt_Z\Suspt_Z)^3\cdots
  }
\end{align}
of the identity functor and consider the fibration sequence $F\rightarrow E\rightarrow B$ in $\SpaceZ$. Since we know that $E,B$ are 1-connected (rel. $Z$) by assumption, this means that we have the homotopical estimates in Propositions \ref{prop:uniformity_estimates_low_dimension} and \ref{prop:estimates_for_spaces} available. With this in mind, let's resolve $E,B$ with respect to the Bousfield-Kan cosimplicial resolution
\begin{align*}
\xymatrix{
  F\ar[r]\ar[d] &
  \tilde{F}^0\ar[d]\ar@<-0.5ex>[r]\ar@<0.5ex>[r] &
  \tilde{F}^1\ar[d]
  \ar@<-1.0ex>[r]\ar[r]\ar@<1.0ex>[r] &
  \tilde{F}^2\ar[d]\cdots\\
  E\ar[r]\ar[d] &
  (\Loopt_Z\Suspt_Z)(E)\ar[d]\ar@<-0.5ex>[r]\ar@<0.5ex>[r] &
  (\Loopt_Z\Suspt_Z)^2(E)
  \ar@<-1.0ex>[r]\ar[r]\ar@<1.0ex>[r]\ar[d] &
  (\Loopt_Z\Suspt_Z)^3(E)\cdots\ar[d]\\
  B\ar[r] &
  (\Loopt_Z\Suspt_Z)(B)\ar@<-0.5ex>[r]\ar@<0.5ex>[r] &
  (\Loopt_Z\Suspt_Z)^2(B)
  \ar@<-1.0ex>[r]\ar[r]\ar@<1.0ex>[r] &
  (\Loopt_Z\Suspt_Z)^3(B)\cdots
  }
\end{align*}
and take homotopy fibers vertically to define the coaugmented cosimplicial diagram of the form $F\rightarrow\tilde F$. By construction the columns are homotopy fiber sequences in $\SpaceZ$, and since $E,B$ are 1-connected (rel $Z$), we know from Theorem \ref{MainTheorem_BK_completion} that
\begin{align*}
  E&\xrightarrow{\wequiv}\holim\nolimits_\Delta (\Loopt_Z\Suspt_Z)^{\bullet+1}(E)\\
  B&\xrightarrow{\wequiv}\holim\nolimits_\Delta (\Loopt_Z\Suspt_Z)^{\bullet+1}(B)
\end{align*}
Since homotopy limits commute with homotopy fibers, it follows that 
\begin{align*}
  F\wequiv\holim\nolimits_\Delta\tilde F
\end{align*}
We want to show that $F\wequiv F^\wedge_{\Loopt_Z\Suspt_Z}$. To get Bousfield-Kan completion into the picture, let's resolve each term in $F\rightarrow\tilde F$ with respect to \eqref{eq:a_nice_resolution} to obtain the cosimplicial resolution of coaugmented cosimplicial diagrams of the form
\begin{align}
\label{eq:nice_array_of_stuff}
\xymatrix{
  (\Loopt_Z\Suspt_Z)^3F\ar[r]^-{(\#)} &
  (\Loopt_Z\Suspt_Z)^3\tilde{F}^0\ar@<-0.5ex>[r]\ar@<0.5ex>[r] &
  (\Loopt_Z\Suspt_Z)^3\tilde{F}^1
  \ar@<-1.0ex>[r]\ar[r]\ar@<1.0ex>[r] &
  (\Loopt_Z\Suspt_Z)^3\tilde{F}^2\cdots\\
  (\Loopt_Z\Suspt_Z)^2F\ar[r]^-{(\#)}\ar@<-1.0ex>[u]\ar[u]\ar@<1.0ex>[u] &
  (\Loopt_Z\Suspt_Z)^2\tilde{F}^0\ar@<-0.5ex>[r]\ar@<0.5ex>[r]\ar@<-1.0ex>[u]\ar[u]\ar@<1.0ex>[u] &
  (\Loopt_Z\Suspt_Z)^2\tilde{F}^1\ar@<-1.0ex>[u]\ar[u]\ar@<1.0ex>[u]
  \ar@<-1.0ex>[r]\ar[r]\ar@<1.0ex>[r] &
  (\Loopt_Z\Suspt_Z)^2\tilde{F}^2\cdots\ar@<-1.0ex>[u]\ar[u]\ar@<1.0ex>[u]\\
  (\Loopt_Z\Suspt_Z)F\ar[r]^-{(\#)}\ar@<-0.5ex>[u]\ar@<0.5ex>[u] &
  (\Loopt_Z\Suspt_Z)\tilde{F}^0\ar@<-0.5ex>[r]\ar@<0.5ex>[r]\ar@<-0.5ex>[u]\ar@<0.5ex>[u] &
  (\Loopt_Z\Suspt_Z)\tilde{F}^1
  \ar@<-1.0ex>[r]\ar[r]\ar@<1.0ex>[r]\ar@<-0.5ex>[u]\ar@<0.5ex>[u] &
  (\Loopt_Z\Suspt_Z)\tilde{F}^2\cdots\ar@<-0.5ex>[u]\ar@<0.5ex>[u]\\
  F\ar[r]^-{(\#)}\ar[u] &
  \tilde{F}^0\ar@<-0.5ex>[r]\ar@<0.5ex>[r]\ar[u]_-{(**)} &
  \tilde{F}^1\ar[u]_-{(**)}
  \ar@<-1.0ex>[r]\ar[r]\ar@<1.0ex>[r] &
  \tilde{F}^2\cdots\ar[u]_-{(**)}
  }
\end{align}
We know by our homotopical estimates (Propositions \ref{prop:uniformity_estimates_low_dimension} and \ref{prop:estimates_for_spaces}) that the coface $(n+1)$-cubes (\cite[3.13]{Blomquist_Harper_integral_chains}) associated to
\begin{align*}
  E&\rightarrow (\Loopt_Z\Suspt_Z)^{\bullet+1}(E)\\
  B&\rightarrow (\Loopt_Z\Suspt_Z)^{\bullet+1}(B)
\end{align*}
are $((\id+1)+1)$-cartesian in $\SpaceZ$, and hence it follows by several applications of \cite[1.6, 1.18]{Goodwillie_calculus_2} that the coface $(n+1)$-cubes associated to $F\rightarrow\tilde F$ are $(\id+1)$-cartesian in $\SpaceZ$. We know, by Proposition \ref{prop:preserving_id_plus_one_cartesian_cubes}, that $(\Loopt_Z\Suspt_Z)$ preserves $(\id+1)$-cartesian $(n+1)$-cubes in $\SpaceZ$ for each $n\geq -1$. Therefore, the coface $(n+1)$-cubes (\cite[3.13]{Blomquist_Harper_integral_chains}) associated to
\begin{align*}
  (\Loopt_Z\Suspt_Z)^k F\rightarrow
  (\Loopt_Z\Suspt_Z)^k \tilde F,\quad\quad
  k\geq 0
\end{align*}
are $(\id+1)$-cartesian in $\SpaceZ$ for each $n\geq -1$, and hence each of the maps
\begin{align*}
  (\Loopt_Z\Suspt_Z)^k F\rightarrow\holim\nolimits_{\Delta^{\leq n}}
  (\Loopt_Z\Suspt_Z)^k \tilde F,\quad\quad
  k\geq 0
\end{align*}
is $(n+2)$-connected. Therefore applying $\holim_\Delta$ horizontally to the maps $(\#)$ induces a weak equivalence, and hence applying $\holim_\Delta$ first horizontally and then vertically produces 
\begin{align*}
  \holim_{\mathsf{vert}}\holim_{\mathsf{horiz}}\eqref{eq:nice_array_of_stuff}\wequiv
  F^\wedge_{\Loopt_Z\Suspt_Z}
\end{align*}
What about the other way? Since the $(**)$ columns have extra codegeneracy maps $s^{-1}$ \cite[6.2]{Dwyer_Miller_Neisendorfer} (by formal reasons: $\Loopt_Z$ commutes with homotopy fibers), applying $\holim_\Delta$ vertically produces \cite[3.16]{Dror_Dwyer_long_homology} the coaugmented cosimplicial diagram
\begin{align*}
\xymatrix{
  F\ar[r] &
  \tilde{F}^0\ar@<-0.5ex>[r]\ar@<0.5ex>[r] &
  \tilde{F}^1
  \ar@<-1.0ex>[r]\ar[r]\ar@<1.0ex>[r] &
  \tilde{F}^2\cdots
}
\end{align*}
and hence applying $\holim_\Delta$ first vertically and then horizontally produces
\begin{align*}
  \holim_{\mathsf{horiz}}\holim_{\mathsf{vert}}\eqref{eq:nice_array_of_stuff}\wequiv
  F
\end{align*}
Hence we have verified that the coaugmentation
\begin{align*}
  F\wequiv F^\wedge_{\Loopt_Z\Suspt_Z}
\end{align*}
is a weak equivalence. The other cases are similar (the estimates are identical). Consider the case of the homotopy pullback square; then $F\rightarrow\tilde F$ is constructed by taking homotopy pullbacks instead of homotopy fibers and the above arguments complete the proof.
\end{proof}

\bibliographystyle{plain}
\bibliography{RetractiveSpaces}

\end{document}